\documentclass[reqno]{amsart}
\usepackage{amssymb}
\usepackage[utf8]{inputenc}
\usepackage[english]{babel}
\usepackage{latexsym,cite,mathrsfs}
\usepackage{amsfonts}
\usepackage{amssymb, amsthm}
\usepackage{xcolor}

\def\W{\accentset{\circ}{W}^{1}}
\usepackage{accents}

\numberwithin{equation}{section}

\newtheorem{thm}{Theorem}[section]
\newtheorem{lem}[thm]{Lemma}
\newtheorem{prop}[thm]{Proposition}

\newtheorem{rem}[thm]{Remark}

\begin{document}
%%%%%%%%%%%%%%%%%%%%%%%%%%%%%%%%%%%%%%
\title[A finding of the maximal saddle-node bifurcation]{A finding of the maximal saddle-node bifurcation  for systems of differential equations}

\author[Y. Il'yasov]{Yavdat Il'yasov}

\address{Institute of Mathematic\\ Ufa Federal Research Centre, RAS\\
	Chernyshevsky str. 112, 450008 Ufa\\ Russia}
\email{ilyasov02@gmail.com}

%%%%%%%%%%%%%%%%%%%%%%%%%%%%%%%%%%%%%%
\begin{abstract}  A variational method is presented for directly finding the bifurcation point of nonlinear equations as the saddle-node point of the extended nonlinear Rayleigh quotient.
The method is applied for solving an open problem on the existence of a maximal saddle-node bifurcation point for  set of positive solutions of system equations with convex-concave type nonlinearity.
\end{abstract}

%%%%%%%%%%%%%%%%%%%%%%%%%%%%%%%%%%%%%%%%%%%%%%%%%%%%%%%%%%%%%%%%%%%%%%%%%%%%%%%%%%%%%%%%%%%%%%%%%%%%%%%%%%%%%%55
\maketitle

\section{Introduction}
It is generally understood that many  phenomena in physics, biology, economics and medicine can be mathematically represented in terms of  bifurcations. Detection of  bifurcations, finding and estimating their characteristic values is important  for many problems in science and industry, including  climate science, epidimology, power system management, molecular biophysics, mesoscopic physics and climatology, and etc., see, e.g., \cite{arnold,  kuzn, seydel}, and references therein.
In a broad sense, the term bifurcation is used  for designating  all sorts of qualitative reorganizations of various entities resulting from a change of the parameters on which they depend (\cite{arnold, Gilmore}).  This  paper is concerned with finding the  saddle-node  bifurcations. To be more specific,  consider the   nonlinear equation 
	\begin{align}
	%\tag{$f$}
	\label{f}
	F(u,\lambda)=0, ~~~~u \in S, ~~\lambda \in \mathbb{R}.
\end{align}
Here $S $ is a given domain in $X$, $X,Z$ are real Banach spaces, $F: X\times \mathbb{R} \to Z$ is a continuously Fr\'echet differentiable map.
We call a solution $(\hat{u},\hat{\lambda}) \in S\times \mathbb{R}$ of \eqref{f} the \textit{saddle-node type bifurcation point}  in $S$    if the following is fulfilled: (i) the null space $N(F_u(\hat{u},\hat{\lambda}))$ of the Fr\'echet derivative $F_u(\hat{u},\hat{\lambda})$  is not  empty, i.e., the operator $F_u(\hat{u},\hat{\lambda})$ is singular; (ii) there is a deleted left neighbourhood $(\hat{\lambda}-\varepsilon, \hat{\lambda})$  of $\hat{\lambda}$ and a neighbourhood $U \subset X$ of $\hat{u}$ such that for each $\lambda \in (\hat{\lambda}-\varepsilon, \hat{\lambda})$  the operator $F_u(u_\lambda,\lambda)$ is non-singular for any solution $u_\lambda \in S\cap U$ of \eqref{f}; (iii) there is a deleted right neighbourhood $(\hat{\lambda},\hat{\lambda}+\varepsilon)$  of $\hat{\lambda}$  in $\mathbb{R}$   such that for any $\lambda \in (\hat{\lambda},\hat{\lambda}+\varepsilon)$ equation \eqref{f} has no solutions in $S\cap U$. In the literature,  the point $(\hat{u},\hat{\lambda})$  is known as a saddle-node bifurcation point (or, equivalently, fold, turning point) (see, e.g., \cite{kielh, keller1}) if the following is also true: (iv) for each $\lambda \in (\hat{\lambda}-\varepsilon, \hat{\lambda})$, the equation has precisely two distinct solutions in $S\cap U$. This definition corresponds to the solutions curve bifurcated to the left from the bifurcation value $\hat{\lambda}$. The definition of the saddle-node bifurcation point with the solutions curve bifurcated to the right from the bifurcation value is similar. 
%It is also assumed that the left and right neighbourhoods of $\hat{\lambda}$ given in problems (ii) and (iv) can be interchanged, respectively. The above definition corresponds to the case when the solutions branches are bifurcate to the left of the bifurcation value $\hat{\lambda}$. It is also allowed when the solutions branches are bifurcate to the right of the value, in this case,  the  left and right neighborhoods  of value $\hat{\lambda}$ that given in issues (ii) and (iv) are interchanged, respectively.

 We are mainly concerned with finding the \textit{maximal saddle-node type bifurcation point} $(u^*,\lambda^*)$ in $S$, which is characterized by the property of having  the maximal bifurcation value  $\lambda^*$ among all possible saddle-node type bifurcations $(\hat{u},\hat{\lambda})$ of \eqref{f} in $S$, namely $\hat{\lambda}\leq  \lambda^*$. 

We shall call $(u^*,\lambda^*)$ the \textit{extremal solution} of \eqref{f} in $S$ if for any $\lambda>\lambda^*$ the equation has no solutions in $S$, while for $\lambda\leq \lambda^*$ the equation  may admit  solutions in $S$. Thus, the maximal saddle-node type bifurcation point can be viewed also the extremal solution. 
This definition has a great deal in common with the well-known concept of an extremal solution, which has been extensively studied since the seminal work of Crandall and Rabinowitz \cite{CranRibinExtr} (see, e.g., \cite{BrezCazMR, BrezVaz, cabre, Cazenave, Dupaigne, Montenegro}).
However, there are examples of extremal solutions that aren't bifurcation points (see, e.g., \cite{BrezVaz,  Dupaigne}), so the concept of an extremal solution is broader than a maximal saddle-node type bifurcation point. Furthermore, our work has somewhat different research objectives from those on extremal solutions.

%
%
%In general, a saddle-node bifurcation usually occurs when stable and unstable equilibria annihilate.
%In this paper, along with finding the saddle-node bifurcation, special attention will be paid to finding the corresponding branch of stable solutions.

%A natural and open question is to know \textit{whether a given equation has a saddle-node bifurcation}. In particular, is it possible to find sufficient general conditions for the existence of a saddle-node bifurcation for a given equation  and selected set $S$? The next substantional question is \textit{how to find the saddle point bifurcations}. Is there a general method that allows direct finding of the bifurcation point $(u^*, \lambda^*)$?

A natural and open question is: 
\begin{description}
	\item[\rm{(q1)}] \textit{Does a given equation possess a saddle-node bifurcation?}
\end{description}
	In particular, is it possible to find sufficient general conditions for the existence of the saddle-node bifurcation of a given equation in a prescribed set $S$? 
The following substantional question is 
\begin{description}
	\item[\rm{(q2)}] \textit{How to find the saddle-node bifurcation point $(u^*,\lambda^*)$? }
\end{description}
In particular, is there a general method that allows direct finding of the bifurcation point $(u^*, \lambda^*)$?

%One can conceptualize the direct determination of the bifurcation point as a system of equations solution:
As an indication of the complexity of this issue, let us note that the direct finding of the bifurcation point in a sense implies solving  the system of equations:
\begin{equation*}
\left\{ \begin{aligned}
  &F(u^*,\lambda^*)  =0,\\
&F_u(u^*,\lambda^*)(v^*)=0, 
\end{aligned}\right.
\end{equation*}
with unknown $(u^*, v^*,\lambda^*) \in S\times X\times \mathbb{R}$.

In the nonlinear analysis and numerical methods, continuation methods by a parameter are commonly used for finding bifurcations. 
We refer to continuation methods as those that use local or global construction of the solutions curve of the equation for finding the bifurcation point located on them. Indeed, the numerical continuation methods \cite{kuzn, seydel} find a bifurcation by producing a set of points $(u_{\lambda_{i}},\lambda_{i})$, $i=1,\ldots,k$ on the solution curve until a singular point occurs, that is, when the operator $F_u(u_{\lambda_{k}},\lambda_{k})$ becomes singular.
%The well-known methods such as  Crandall-Rabinowitz’s theorem \cite{CranRibin}, Krasnoselskii's theorem \cite{kras}, Leray-Schauder continuation theorem \cite{leray}, Rabinowitz’s global bifurcation theorem \cite{rabin}, Vainberg–Trenogin's branching method \cite{tren}  provide powerful tools for finding of bifurcations to nonlinear equations.  These methods include a global analysis of the qualitative and topological structure of the set of solutions  $\mathcal{J}:=\overline{\{(u,\lambda):~u\neq 0, F(u,\lambda)=0 \}}$ and exploring the limit point (a priori bifurcation point) $(\hat{u},\hat{\lambda})$ on a curve of solutions (see, e.g.,  surveys \cite{Ize, Mawlin, Drabek}).
%Nevertheless, the application of these methods can cause difficulties in a number of situations; in some cases (for example, for systems of equations),  the construction of a curve of solutions up to the bifurcation point is a complicated problem in itself.
%In many cases, the necessary global qualitative and topological analysis of the solution set is difficult or does not give sufficiently detailed information about the solution branches and bifurcation points. In addition, we are not aware of any examples in which these methods have been applied to find the maximal bifurcation point.
The famous methods, such as  Crandall-Rabinowitz’s theorem \cite{CranRibin}, Krasnoselskii's theorem \cite{kras}, Leray-Schauder continuation theorem \cite{leray}, Rabinowitz’s global bifurcation theorem \cite{rabin}, Vainberg–Trenogin's branching method \cite{tren}, etc., all provide powerful tools for detection of bifurcations of wide classes of nonlinear equations. These methods are based on the analysis of the qualitative and topological structure of the set of solutions $\mathcal{J}:=\overline{\{(u,\lambda):~u\neq 0, u \in S, F(u,\lambda)=0 \}}$ and examination of the limit point (a priori bifurcation point) $(\hat{u},\hat{\lambda})$ on the solution curve (see, e.g., \cite{Ize, Mawlin, Drabek}). 
Nevertheless, the application of these methods can cause difficulties in a number of situations. For instance, in some cases (for example, for systems of equations), the construction of a curve of solutions up to the bifurcation point $(\hat{u},\hat{\lambda})$ is a complicated problem in itself.
In other cases, the necessary global qualitative and topological analysis of the solution set is difficult or does not give sufficiently detailed information about the solution branches and bifurcation points.
A different approach to solving the above problems was proposed in \cite{IlyasFunc}, which relies on the so-called \textit{extended functional method}. This approach has been further developed in   \cite{ilBifChaos}, where complete solutions for the above problems (q1) and (q2)  were obtained in the finite-dimensional cases. In particular, general sufficient conditions for the existence of saddle-node bifurcations of abstract finite-dimensional equations were found, and at the same time, a direct method for finding the saddle-node bifurcation point has been substantiated. Furthermore, the extended functional method entails the emergence of a fundamentally new approach for numerical  finding the bifurcation \cite{IlIvan1, IlIvan2, Salazar}.  The advantage of this approach is that a bifurcation point can be directly identified by applying gradient descent to find a minimum of a given function. On the other hand, the commonly used numerical continuation methods (see, e.g. \cite{kuzn, seydel}) require computing a sequence of points on the solutions curve up to the detection of a bifurcation point.

The extended functional method has been applied to finding the saddle-node point of an elliptic problem with an indefinite sign nonlinearity in \cite{IlyasFunc}.
The saddle-node bifurcation point there was, however, obtained in essence by the continuation method,  which entails difficulty in generalizing it to the systems of differential equations (see, e.g, \cite{BobkovIlyasov}).

The purpose of this paper to demonstrate that the extended functional method can be applied to obtain complete solutions to problems (q1) and (q2).

\section{Main result}\label{Sec:Main}
We consider the following system of equations with convex-concave type nonlinearities
\begin{equation}
%\tag{$\mathcal{N}$}
\label{p}
\left\{ \begin{aligned}
  -\Delta u_i  &= \lambda  |u_i|^{q-2}u_i + g_i(x, u)  ,  &&x \in \Omega, \\[0.4em]
~u_i |_{\partial \Omega} &= 0,~~~i=1,\ldots, m. 
\end{aligned}\right.
\end{equation}
%\begin{equation}
%\tag{$\mathcal{N}$}
%\label{p}
%\left\{ \begin{aligned}
  %-\Delta u_i  &= \lambda  |u_i|^{q-2}u_i + g_i(x, u)  ,  &&x \in \Omega, \\[0.4em]
%~u_i |_{\partial \Omega} &= 0, \\[0.4em]
%u_i>&0,~~~i=1,\ldots m.
%\end{aligned}\right.
%\end{equation}
 Here  $\Omega$ is a bounded domain in $\mathbb{R}^d$ with $\partial \Omega \in C^2$, $d \geq 1$, $\lambda \in \mathbb{R}$,  $u:=(u_1,\ldots, u_m)$, $1<q<2$, $g_i(\cdot,u) $, $i=1,...,m$ are H\"older continuous  functions in $\Omega$, $\forall u \in \mathbb{R}^m$, and $g_i(x,\cdot) \in C^1(\mathbb{R}^m,\mathbb{R})$,   $g_i(x,0)=0$, $g_i(x,u)\geq 0$, $x \in \Omega$, $u \in \mathbb{R}^m$, $i=1,...,m$, with primitive $G(x,u)$,  that is, $g_i(x,u)=G_{u_i}(x,u)$ , $x \in \Omega$, $u \in \mathbb{R}^m$, $i=1,...,m$. Furthermore, we assume

\par 
	$(g_1):$\, $ \exists  c_1,c_2 > 0$ and $  \gamma_1, \gamma_2	\in (2,2^*)$, $\gamma_1\leq \gamma_2$,
	$$
	g_{i,u_j}(x,u)u_ju_i \leq c_1|u|^{\gamma_1} +c_2|u|^{\gamma_2}, x \in \Omega,~~ u \in \mathbb{R}^m;
	$$
		\par
		$(g_2):$ \, $\exists \theta>2$,  $ \theta G(x,u)\leq g_i(x,u)u_i$, $x \in \Omega$,   $u \in \mathbb{R}^m$;
\medskip
\par
		$(g_3):$  $ (q+1)g_i(x,u)u_i\leq g_{i,u_j}(x,u)u_ju_i $, $x \in \Omega$,   $u \in \mathbb{R}^m$;
		\medskip
	\par
$(g_4):$\, there exist a domain $\Omega^+ \subseteq \Omega$ and $
 R>0$: 
	\begin{equation*}
\frac{\sum_{i=1}^mg_i(x,u)}{\sum_{i=1}^mu_i} >\lambda_1(\Omega^+),~~ u \in \mathbb{R}^m_+, ~~|u|>R,~ x \in \Omega^+.	
	\end{equation*}
Hereafter, $\lambda_1(\Omega^+)$ denotes the principal eigenvalue of the operator $(-\Delta)$ in $\W_2(\Omega^+)$.		
 Throughout this paper the summation convention is in place: we sum over any index that appears twice. We denote  $|v|^p=\sum_{i=1}^m|v_i|^p$, $p\geq 1$, $|\nabla v|^2:=(\nabla v_i, \nabla v_i)$, for $v \in (\W_2(\Omega))^m$, where $\nabla:=(\partial/\partial x_1,\ldots, \partial/\partial x^d)$, $(\cdot,\cdot)$ states the scalar product in $\mathbb{R}^d$. Note that $(g_1)$ implies 	 that $ \exists c_1',c_2' > 0$: $
	 0\leq \sum_{i=1}^mg_i(x,u) \leq c_1'|u|^{\gamma_1-1} +c_2'|u|^{\gamma_2-1}~~\mbox{in}~ \Omega,~~u \in \mathbb{R}^m$. 	

A model  example for \eqref{p} in the scalar case, i.e., $m=1$, is the Ambrosetti–Brezis–Cerami problem
\cite{ABC} with concave–convex nonlinearity  
\begin{equation}
\label{ps}
	 -\Delta u= \lambda  |u|^{q-2}u + |u|^{\gamma-2}u, ~~~u |_{\partial \Omega} = 0,
\end{equation}
where $1<q<2$, $2+q<\gamma$. By \cite{ABC, CazeEscobedo}, there exists an extremal  value $\lambda^*>0$ such that for any $\lambda \in (0,\lambda^*]$, \eqref{ps} has a stable positive   solution $u_\lambda$, while for  $\lambda>\lambda^*$, \eqref{ps} does not admit  weak positive solutions. The solutions $u_\lambda$ for $\lambda \in (0,\lambda^*)$ can be obtained by the super-sub solution method, while the existence of $u_{\lambda^*}$ is deduced as a limit point of $(u_\lambda)$ (see, e.g., \cite{ABC, CazeEscobedo}). However, this result has difficulty being generalized to systems of equations like \eqref{p}.  Note that the super-sub solution method for a system of equations differs considerably from that which is used for a scalar equation.

The  existence of positive solutions of \eqref{p}  for $\lambda \in (0,\lambda_0)$ with sufficiently small $\lambda_0$ can be obtained  by standard ways, see e.g., \cite{Brown, BobkovIlyasov, BobkovIlyasov1, wu}, or see below Section \ref{sec:loc}. Furthermore, below we show that there exists an extremal value $\lambda^*_S>0$ such that \eqref{p} does not admit stable positive solutions if $\lambda>\lambda^*_S$. However, finding an extremal solution $u_{\lambda^{*}_S}$ for the system of equations by passing to a limit  has trouble due to the difficulty of finding a solution curve for the system of equations over the entire interval $(0,\lambda^*_S)$.

%The method of subsolution–supersolution for system of equations differs substantially from what happens in the case of single equations.
%
%The method of subsolution–supersolution for system of equations differs substantially from what happens in the case of single equations. It is worth noticing that for systems one cannot define separately the notions of subsolution and supersolution, but the notion of subsolution–supersolution (or trapping region).
Let us state our main results. 
Denote $\mathcal{W}:=(\W_2(\Omega))^m$, and define
\begin{align*}
&\mathcal{W}^+:=\{u \in \mathcal{W}:~u_i \geq  0~\mbox{a.e.}~\Omega,~i=1,\ldots,m\},\\
		&\mathcal{W}^+_c:=\{u \in  (C^1(\overline{\Omega}))^m\cap \mathcal{W}^+: ~~u_i > 0 ~\mbox{in}~\Omega, ~i=1,\ldots,m \},\\
	&\Sigma(u):=\{v \in \mathcal{W}: \int u_iv_i   \neq 0\}, ~~~~ u \in \mathcal{W}^+_c. 
\end{align*}
It is easily seen that $\mathcal{W}^+ \subset \Sigma(u)$, i.e., $\Sigma(u)\neq \emptyset$, for any $u \in \mathcal{W}^+_c$, and  $\Sigma(u)$ is an open domain in $\mathcal{W}$.

 By a weak solution of (\ref{p}) we   mean a critical point $u_\lambda \in\mathcal{W}$ of the Euler-Lagrange functional
\begin{equation}\label{pb11c}
\Phi(u,\lambda) = \frac{1}{2} \int |\nabla u|^2 dx - \lambda\frac{1}{q} \int |u|^{q} dx -
\int G(x,u)dx,
\end{equation}
that is,   $F(u_\lambda,\lambda):=\Phi_u(u_\lambda,\lambda):=(\Phi_{u_1}(u_\lambda,\lambda),\ldots, \Phi_{u_m}(u_\lambda,\lambda))=0$. Note that  by $(g_1)$, $\Phi(u,\lambda)$ is a continuously Fr\'echet differentiable functional on $\mathcal{W}$.

We define the  \textit{extended Rayleigh quotient} as follows: 
\begin{align*}
\notag
\mathcal{R}&(u, v) :=  \frac{\int (\nabla u_i, \nabla v_i) \, 
- \int g_i(x, u)    v_i \,   }
{\int  u_i^{q-1}v_i\, },~~u\in \mathcal{W}^+_c, v\in \Sigma(u).
\end{align*}
Note that if $\mathcal{R}(u, v)=\lambda$, $v\in \Sigma(u)$, $u \in \mathcal{W}^+_c$, then 
\begin{align*}
	&\mathcal{R}_v(u, v)=0~~\Leftrightarrow~~\Phi_u(u,\lambda)=0,\\
	&\mathcal{R}_u(u, v)=0~~\Leftrightarrow~~\Phi_{uu}(u,\lambda)(v)=0.
\end{align*}
 %if $\mathcal{R}(u, v)=\lambda$, then $\displaystyle{\mathcal{R}_v(u, v)=\frac{1}{\int  u_i^{q-1}v_i}\Phi_u(u,\lambda)}$
%

Introduce the following subset
\begin{align*}
		\mathcal{W}^+_S:=\{u \in \mathcal{W}^+_c:~~\min_{x' \in \partial \Omega}|\partial u_i(x')/\partial \nu(x')| >0,~i=1,\ldots,m,~~~ \delta(u)\geq 0\},
\end{align*}
where $\nu(x')$ denotes the outward unit normal at $x' \in \partial \Omega$, and 
\begin{equation}\label{eigenProb}
	\delta(u):=\inf_{\phi \in \mathcal{W}}\frac{\int |\nabla \phi|^2\, -\int g_{i,u_{j}}(x, u)\phi_j\phi_i \, -(q-1) \mathcal{R}(u,u) \int u^{q-2}|\phi|^2 }
		{\int |\phi|^2 \, }.
\end{equation}
 Below we show (see  Remark \ref{remF}) that  for any $u \in \mathcal{W}^+_S$ the second derivative $\Phi_{uu}(u,\lambda)$ is well-defined on $\mathcal{W}\times \mathcal{W}$. 
	
Note that for $u \in \mathcal{W}^+_S$ with $\lambda=\mathcal{R}(u,u)$,
	$$
	(\Phi_{uu}(u,\lambda)(\phi))_i=-\Delta \phi_i - g_{i,u_j}(x,u)\phi_j -\lambda (q-1)  u_i^{q-2}\phi_i 
	\quad i=1,\ldots,m,~~\phi \in \mathcal{W},
$$
and thus, $\delta(u)$ is the principal eigenvalue
  of the operator $\Phi_{uu}(u,\lambda)$, and 
\begin{equation}\label{conStab}
	\delta(u)\geq 0~~\Leftrightarrow~~ \Phi_{uu}(u,\lambda)(\phi,\phi)\geq 0,~~\forall \phi \in \mathcal{W},~~~ \lambda=\mathcal{R}(u,u).
\end{equation}

We call a solution $u_\lambda \in \mathcal{W}^+_c$ of \eqref{p}  \textit{stable} if $\delta(u_\lambda)\geq 0$, and \textit{asymptotically stable} if $\delta(u_\lambda)>0$ (cf. \cite{ABC, cabre, Dupaigne}).

Introduce the following \textit{ nonlinear generalized Collatz-Wielandt value} \cite{Collatz, Wielandt} 
\begin{equation}\label{MainB}
	\lambda^{*}_S:= \sup_{u\in \mathcal{W}^+_S}\inf_{v\in \Sigma(u)}\mathcal{R}(u, v).
\end{equation}
We  adhere to the convention that $\lambda^{*}_S=-\infty$ if $\mathcal{W}^+_S =\emptyset$.

Our main result on the existence of a maximal saddle-node type bifurcation point  is as follows.
\begin{thm}\label{thmM} 
Assume that there holds $(g_1)$ - $(g_4)$, $1<q<2$.
Then $0< \lambda^{*}_S < +\infty$, and 
 \par 
$(1^o)$  for any $\lambda>\lambda^{*}_S$, system \eqref{p} has no stable positive   solutions;
 \par 
$(2^o)$ 
for $\lambda=\lambda^{*}_S$, system \eqref{p} has a stable positive solution $u_{\lambda^{*}_S} \in  \mathcal{W}^+_S$. Moreover, the operator $\Phi_{uu}(u_{\lambda^{*}_S},{\lambda^{*}_S})$ is singular. 
\par 
 There exists $\bar{\lambda} \in (0,\lambda^{*}_S]$ such that 
for every $\lambda \in (0,\bar{\lambda})$ system \eqref{p} admits an asymptotically stable weak positive solution $u_\lambda \in  \mathcal{W}^+_S$, in particular,  $\Phi_{uu}(u_\lambda,\lambda)$ is non singular.

\end{thm}

\begin{rem}
We believe that  $\bar{\lambda}=\lambda^*_S$, and $(u_{\lambda^{*}_S},\lambda^{*}_S)$ is indeed a saddle-node bifurcation point.
\end{rem}
\begin{rem}\label{rem1}
	Any weak non-negative solution  $u_\lambda \in \mathcal{W}$ to \eqref{p} is positive and belongs  $u_\lambda \in \mathcal{W}^+_S\cap (C^{2}(\Omega))^m$.
	Indeed, the standard bootstrap argument and  Sobolev's embedding theorem (see, e.g., \cite{ Struwe}) entail that  $u_\lambda \in (L^\infty(\Omega))^m$. 
	Therefore, by the $L^p$-regularity results in \cite{adams, Giltrud}, $u_\lambda \in (W^{2,p}(\Omega))^m$ for any $1 < p < \infty$ and thus, by Sobolev's embedding theorem, $u_\lambda \in (C^{1,\alpha}(\overline{\Omega}))^m$ for any $\alpha \in (0, 1)$. Moreover, since $g_i(x,u_\lambda(x)) $, $i=1,...,m$ are H\"older continuous  functions in $\Omega$,  the Schauder estimates and  the Hopf boundary maximum
principle \cite{trudin} imply that $u_\lambda \in (C^{2}(\Omega))^m$ and  $u_{\lambda}\in \mathcal{W}^+_S$.
\end{rem}

\begin{rem}
The extended functional method has been used to solve various theoretical problems from  the nonlinear partial differential equations   including problems which are not directly related to the finding of bifurcations (see, e.g., \cite{BobkovIlyasov,BobkovTanaka, IlD, IlRunst}).	
\end{rem}

The rest of the paper is organised as follows. Section \ref{sec: thm0} presents main Lemmas.   
In Sect. \ref{sec:loc}, we prove using the Nehari manifold method the local existence of positive solutions with respect to the parameter $\lambda$ of the problem. As a result, we can estimate the  maximal saddle-node type value $\lambda^*_S$ from below.  Section \ref{sec: profthm1} is devoted to the proof of the main result on the existence of the maximal saddle-node type bifurcation point $(u_{\lambda^{*}_S},\lambda^{*}_S)$. In Appendix, we present a proof of a version of  Ekeland's principal for smooth functional.

Hereinafter, we use the standard notation $L^p:=L^p(\Omega)$
for the Lebesgue spaces, $1 \leq p \leq +\infty$, and denote by $\|\cdot\|_p$
the associated norm. By $\accentset{\circ}{W}_2^{1}:=\accentset{\circ}{W}_2^{1}(\Omega)$ we denote the standard Sobolev space, endowed with the norm $\|u \|_{1,2}=(\int |\nabla u|^2)^{1/2}$, $2^*$ stands for the critical Sobolev exponent, i.e., $2^*=2d/(d-2)$ if $d >2$, $2^*=+\infty$ if $d\leq 2$. In what follows, for simplicity we assume that $d>2$. For  a normed space $W$ we denote by $W'$ the  dual space. For $F \in C^1( W )$, $u \in W$ we denote $F'(tu):=d F(tu)/dt$, $t\in \mathbb{R}$ and $F'(u)=F'(tu)|_{t=1}\equiv F_u(u)(u)$.

\section{Main Lemmas}\label{sec: thm0}
Let $\lambda>0$. Consider the following subset of the Nehari manifold
$$
\mathcal{N}_\lambda^+:=\{u \in\mathcal{W}\setminus 0:~\Phi'(u,\lambda)=0,~~\Phi''(u,\lambda)\geq 0\}.
$$
\begin{prop}\label{prop1} Let $\lambda \in \mathbb{R}$, $(g_1)-(g_3)$ hold true.
	The set $\mathcal{N}_\lambda^+$ is bounded in $\mathcal{W}$, i.e., there exists a constant $C\in (0, +\infty)$ such that $\|u\|_{\mathcal{W}}<C$, $\forall u \in \mathcal{N}_\lambda^+$. Moreover, $\Phi(u,\lambda)<0$, $\forall u \in \mathcal{N}_\lambda^+$.
\end{prop}
\begin{proof} Let $u \in \mathcal{N}_\lambda^+$. Since $\Phi'(u,\lambda)=0$, 
	\begin{align*}
	%\label{CoerC1}
			\Phi(u,\lambda)=\frac{(\theta-2)}{2} \| u\|^2_{\mathcal{W}} -\frac{\lambda(\theta-q)}{q}& \|u\|^{q}_q-\nonumber\\
			&  \int(\theta G(x,u)- g_i(x,u)u_i )dx, 
		\end{align*}
where $\theta>2$ as in $(g_2)$. 
	Hence, by $(g_2)$ and  Sobolev's inequalities 
	\begin{equation}\label{CoerC}
		\Phi(u,\lambda)\geq  \frac{(\theta-2)}{2} \|u\|_{\mathcal{W}}^2-\frac{\lambda(\theta-q)}{q} \|u\|_{\mathcal{W}}^q,~~n=1,\ldots.
		\end{equation}
	Since $\Phi''(u,\lambda)\geq 0$, $\Phi'(u,\lambda)=0$,
$$
	\lambda(2-q)\|u\|^{q}_q  -
	\int(g_{i,u_j}(x,u)u_iu_j - g_i(x,u)u_i )dx\geq 0. 
	$$	
	Notice
	\begin{align*}\label{CoerC1}
			\Phi(u,\lambda)=-\lambda \frac{(2-q)}{2q} \| u\|^q_q -\int(G(x,u)-\frac{1}{2} g_i(x,u)u_i )dx. 
		\end{align*}
		Hence by $(g_3)$
		\begin{align*}
			\Phi(u,\lambda)\leq -\frac{1}{2q}\int(g_{i,u_j}(x,u)u_iu_j &- g_i(x,u)u_i )dx\\
			&-\int(G(x,u)-\frac{1}{2} g_i(x,u)u_i )dx=\\
			-\int( G(x,u)+\frac{1}{2q}&(g_{i,u_j}(x,u)u_iu_j-(q+1) g_i(x,u)u_i ) )dx< 0. 
		\end{align*}
From this and \eqref{CoerC} it follows $C \|u\|_{\mathcal{W}}^q \geq   \|u\|_{\mathcal{W}}^2$, $n=1,\ldots$, where $C \in (0,+\infty)$ does not depend on $u \in \mathcal{N}_\lambda^+$. Since $2>q$, this completes the proof. 
\end{proof}

%Here we used the inequality $g_i(x,u)u_i \geq 0$, $x \in \Omega$, $u \in \mathbb{R}^m$ which follows from $(g_2)$ and the assumption 	 $G(x,u) \geq 0$, $x \in \Omega$, $u \in \mathbb{R}^m$.

\begin{lem}\label{lem1} Suppose $u^0 \in \mathcal{W}^+_c$ such that
$$
-\infty<\lambda_0:=\inf_{v \in \Sigma(u^0)}\mathcal{R}(u^0 ,v)<+\infty.
$$
Then  
$u^0$ is a weak solution of \eqref{p} for $\lambda=\lambda_0$.

\end{lem}
\begin{proof}
Let 	 $v^k \in \Sigma(u^0)$, $k=1,\ldots$, such that
$$
\lambda_k\equiv \mathcal{R}(u^0, v^k) \to  \inf_{v \in \Sigma(u^0)}\mathcal{R}(u^0,v)\equiv\lambda_0~~\mbox{as}~~ k \to +\infty.
$$
Since $\mathcal{R}(u, v)=\mathcal{R}(u, s v)$, $\forall s \in \mathbb{R}\setminus 0$, $\forall v \in \Sigma(u)$, $\forall u  \in \mathcal{W}$, we may assume that
\begin{equation}\label{Norm}
	\int (u_{i}^0)^{q-1}v^k_i=1,~~ k=1,\ldots.
\end{equation}
Calculate
\begin{align*}
	\mathcal{R}_{v}&(u^0, v^k)({\xi})=\frac{\int(\nabla u^0_i, \nabla \xi_i)-\int g_i(x, u^0)\xi_i-\mathcal{R}(u^0, v^k)\int (u_{i}^0)^{q-1}\xi_i}{\int (u_{i}^0)^{q-1}v^k_i },\\
	\mathcal{R}_{vv}&(u^0, v^k)({\xi}, {\zeta})=\\
	&-\frac{\left(\int(\nabla u^0_i, \nabla \xi_i)-\int g_i(x, u^0)\xi_i-\lambda_k\int (u_{i}^0)^{q-1}\xi_i\right)\cdot \int (u_{i}^0)^{q-1}\zeta_i}{(\int (u_{i}^0)^{q-1}v^k_i )^2 }-\\
	& \frac{\left(\int(\nabla u^0_i, \nabla \zeta_i)-\int g_i(x, u^0)\zeta_i-\lambda_k\int (u_{i}^0)^{q-1}\zeta_i\right)\cdot \int (u_{i}^0)^{q-1}\xi_i}{(\int (u_{i}^0)^{q-1}v^k_i )^2 },~\zeta, \xi \in \mathcal{W}.
	\end{align*}
Let $\phi \in \mathcal{W}$, $\|\phi\|_{\mathcal{W}}=1$. Using \eqref{Norm} and the H\"older and Sobolev inequalities one can see that
\begin{align}\label{nomina}
	|\int (u_{i}^0)^{q-1}(v^k_i+\tau {\phi}_i)|=&|1+\tau \int (u_{i}^0)^{q-1} {\phi}_i |\geq   1-a_0|\tau|,
\end{align}
where $ a_0 \in (0, \infty)$ does not depend on   ${\phi}$ and $k=1,\dots$. Hence   $v^k+\tau {\phi} \in \Sigma(u^0)$ for any $k=1,\dots$ and $\tau$ such that $|\tau|<\tau_0:=1/a_0$.

By \eqref{nomina}  we have
\begin{align}\label{DDEst}
	&\|\mathcal{R}_{vv}({u}^0,  v^k+\tau \phi)\|_{(\mathcal{W}\times \mathcal{W})'}=\frac{2}{|\int (u_{i}^0)^{q-1} (v^k_i+\tau {\phi}_i) |^2}\times 	\\
	&\sup_{\xi, \zeta \in \mathcal{W}}\frac{|\left(\int(\nabla {u}^0_i, \nabla \xi_i)-\int g_i(x,u^0)\xi_i-\lambda_k\int (u_{i}^0)^{q-1}\xi_i\right)\cdot \int  (u_{i}^0)^{q-1}\zeta_i|}{\|\xi_i\|_{\W_2} \|\zeta_i\|_{\W_2}}\leq \nonumber\\
		&\frac{2}{(1-a_0|\tau|)^2}\left(\sum_{i=1}^m\|-\Delta {u}^0_i -g_i(x,u^0)- \lambda_k (u_{i}^0)^{q-1}\|_{(\W_2)'} \right)\|{u}^0\|_{\mathcal{W}}
	\leq \frac{C_0}{(1-a_0|\tau|)^2}, \nonumber
\end{align}
where $ C_0 \in (0, \infty)$ does not depend on $k=1,\dots$. We thus may apply Theorem \ref{thm:Ek} to the  functional $F(v):=\mathcal{R}({u}^0, v)$
defined in the open domain   $V:= \Sigma(u^0)\subset \mathcal{W}$, and  $F \in C^2(\Sigma(u^0))$. Indeed, 
\eqref{DDEst} implies \eqref{DDRstG}, while by \eqref{nomina} there holds \eqref{DDRstG2}. Thus,  we have
$$
\epsilon_k:=\|\mathcal{R}_{v}(u^0, v^k)\|_{\mathcal{W}'} \to 0~~~\mbox{as}~~~k \to +\infty,
$$
which by \eqref{Norm} yields: 
\begin{equation*}
	|\int(\nabla u^0_i, \nabla \xi)-\int g_i(x,u^0)\xi-\lambda_k\int (u_{i}^0)^{q-1}\xi|\leq \epsilon_k \|\xi\|_{\mathcal{W}},~~ ~\forall \xi \in \mathcal{W}.
\end{equation*}
$i=1,\ldots,m$.
Now passing to the limit as $k \to +\infty$ we obtain \eqref{p}. 
%, and therefore $\lambda_0\equiv\mathcal{R}(u^0 ,v)$, $\forall v \in \Sigma(u^0)$

\end{proof}
\begin{lem}\label{lemM} 
Assume that there holds $(g_1)$-$(g_3)$ and $1<q<2$.
Suppose $0<\lambda^{*}_S<+\infty$. Then 
\par 
$(1^o)$  for any $\lambda>\lambda^{*}_S$, system \eqref{p} has no weak non-negative solutions;
 \par 
$(2^o)$ 
for $\lambda=\lambda^{*}_S$, system \eqref{p} has a  positive solution $u_{\lambda^{*}_S} \in  \mathcal{W}^+_S$. 
\end{lem}

\begin{proof}

Let us show $(1^o)$. Suppose by contradiction that for $\lambda>\lambda^{*}_S$ there exists a weak solution $u_\lambda \in \mathcal{W}^+_S$ of \eqref{p}. Then by Remark \ref{rem1}, 
$u_\lambda \in \mathcal{W}^+_S$. Hence \eqref{MainB} yields $\inf_{v\in \Sigma(u_\lambda)}\mathcal{R}(u_\lambda, v)<\lambda$, and therefore there exists $v \in \mathcal{W}\setminus 0$ such that
$$
\int(\nabla u_{\lambda,i}, \nabla v_i)-\int g_i(x, u_\lambda) v_i-\lambda\int u_{\lambda,i}^{q-1}v_i<0,
$$
which contradicts \eqref{p} and we thus get assertion $(1^o)$.

Let us prove $(2^o)$. Since $-\infty<\lambda^{*}_S<+\infty$,  the following alternative holds:
\begin{description}
	\item[A)] there exists $u_{\lambda^{*}_S}\in \mathcal{W}^+_S$ such that 
	\begin{equation*}
	\inf_{v \in \Sigma(u_{\lambda^{*}_S})}\mathcal{R}(u_{\lambda^{*}_S},v)=\lambda^{*}_S,
\end{equation*}
\end{description}
or/and
\begin{description}
\item[B)]  there exists a sequence 
$u^n \in \mathcal{W}^+_S$, $n=1,\ldots$, such that 
$$
\lambda_n:=\lambda(u^n):=\inf_{v \in \Sigma(u^n)}\mathcal{R}(u^n,v) \to \lambda^{*}_S~~\mbox{as}~~n\to +\infty.
$$ 
\end{description}

By Lemma \ref{lem1},  case \textbf{A)} implies the existence of a weak solution $u_{\lambda^{*}_S}$ of \eqref{p}.

Suppose \textbf{B)}. Then by Lemma \ref{lem1}, 
\begin{equation}\label{BEq1}
	 -\Delta u_i^n = \lambda_n  |u_i^n|^{q-2}u_i^n+ g_i(x, u^n),~~ i=1,\ldots,m, ~~n=1,\ldots.
\end{equation}
Proposition \ref{prop1} implies that  sequence $\|u^n\|_{\mathcal{W}}$ is bounded. 
Hence by the Banach–Alaoglu theorem  there exists a subsequence (again denoted by $(u^n)$) such that
\begin{equation}\label{wse}
	u^n \to u_{\lambda^{*}_S}~~\mbox{weakly in $\mathcal{W}$} ~~\mbox{and strongly in}~~ (L^p)^m,~~1\leq p< 2^* 
\end{equation}
$\mbox{as }~~n\to +\infty$ for some $ u_{\lambda^{*}_S}\in \mathcal{W}^+$.

To show that $ u_{\lambda^{*}_S}\neq 0$, we  use the following lemma from \cite{ABC}
%{\it Lemma}
\begin{lem} \label{ABCLem}
Assume that $f(t)$ is a function such that $t^{-1}f(t)$ is decreasing for $t>0$, Let $v$ and $w$ satisfy:
$u>0$,$w>0$ in $\Omega$, $v=w=0$ on $\partial \Omega$, and
\begin{align*}
	-\Delta w\leq f(w),~~-\Delta u\geq f(u), ~~\mbox{in}~~ \Omega.
\end{align*}
	Then $u\geq w$.
\end{lem}
Let $w_\lambda$ be a unique positive solution of 
\begin{equation}
\label{q}
\left\{ \begin{aligned}
-\Delta w_\lambda=&\lambda w_\lambda^{q-1}~~\mbox{in}~~ \Omega,\\
w_\lambda |_{\partial \Omega}&= 0.
\end{aligned}\right.
\end{equation}
Note that $w_{\lambda}\equiv (\lambda)^{1/(2-q)}w_1$.

By the assumption  $g_i(x,u)\geq 0$, $x \in \Omega$, $i=1,\ldots,m$, $u \in \mathbb{R}$, and therefore,
$$ 
-\Delta u_i^n  \geq  \lambda_n  (u_i^n)^{q-1} ~~  in ~~ \Omega,~~i=1,\ldots,m,~~ n=1\ldots.
$$
Hence by Lemma  \ref{ABCLem},
\begin{equation}\label{ineqABC}
	u_i^n \geq w_{\lambda_n}\equiv (\lambda_n)^{1/(2-q)}w_1, \quad i=1,\ldots,m,
\end{equation}
which yields $u_{\lambda^{*}_S, i}\geq w_{\lambda^{*}_S}>0$, $i=1,\ldots,m$.

 Now passing to the limit in \eqref{BEq1} as $n\to +\infty$ we obtain
\begin{equation*}
	-\Delta u_{\lambda^{*}_S,i}  = \lambda^{*}_S  |u_{\lambda^{*}_S,i}|^{q-2}u_{\lambda^{*}_S,i} + g_i(x, u_{\lambda^{*}_S}),~~x \in \Omega,~~i=1,\ldots,m.
\end{equation*}
To  prove that $u_{\lambda^{*}_S} \in\mathcal{W}_S^+$, it is sufficient to show that 
\begin{align}\label{ConvDR}
	\Phi_{uu}(u^n,\lambda_n)(\phi,\phi) \to
\Phi_{uu}(u_{\lambda^{*}_S},\lambda^{*}_S)(\phi,\phi)~ \mbox{as}~~ n\to +\infty&,\quad\forall \phi \in \mathcal{W}.
\end{align}
Since $-\Delta w_{1} \geq 0$ in $\Omega$, the Hopf boundary maximum principal yields $\partial w_{1}/\partial \nu <0$ on $\partial \Omega$. This by \eqref{ineqABC} implies that $u^n(x)\geq c \mbox{d}(x):=c\mbox{ dist}(x,\partial \Omega)$ in $\Omega$ for some $c >0$  which does not depend on $ x \in \overline{\Omega}$ and $n=1,\ldots$.

	Note that if $u \in \W_2$ and $u(x)\geq c \mbox{d}(x)$ in $\Omega$ for some $c >0$, then 
	$$
	\int u^{q-2}\phi\psi  \, dx =\int u^{q-1}\left(\frac{\phi}{u}\right) \psi \, dx\leq \frac{1}{c}\|u\|^{q-1}_{2^*-\kappa}\cdot\|\frac{\phi}{\mbox{d}(\cdot)}\|_2\cdot \|\psi\|_{p}, ~~~\forall \phi, \psi\in \W_2,
	$$
	where $p=2\cdot(2^*-\kappa)/(2^*-\kappa-2(q-1))$, $0<\kappa<2^*-2(q-1)$. By the Hardy inequality $\|\phi/\mbox{d}(\cdot)\|_2\leq C\|\phi\|_{1,2}$,  $\forall \phi \in \W_2$, and therefore,
		\begin{equation}\label{welldef}
		\int u^{q-2}\phi\psi \, dx \leq C \|u\|^{q-1}_{2^*-\kappa}\|\phi\|_{1,2}\|\psi\|_{1,2}, ~~~\forall \phi,\psi\in \W_2,
	\end{equation}
for some $C<+\infty$ which does not depend on $\phi, \psi\in \W_2$ and $u$.   Thus, by \eqref{wse}, we have $\int (u^n_i)^{q-2}\phi_i^2  \to
\int u_{\lambda^*_S,i}^{q-2}\phi_i^2 $ as $d\to+\infty$. Similarly, by $(g_1)$, $\int g_{i,u_j}(x,u^n)\phi_i^2 \to \int g_{i,u_j}(x,u_{\lambda^*_S})\phi_i^2$  as $n\to+\infty$, $\forall \phi \in \mathcal{W}$, $j=1,\ldots,m$. Hence we get  \eqref{ConvDR}  and therefore, $\Phi_{uu}(u_{\lambda^{*}_S}, {\lambda^{*}_S})(\phi, \phi)\geq 0$, $\forall \phi \in \mathcal{W}$.
\end{proof}
\begin{rem}\label{remF}
	Note that if  $u \in \mathcal{W}^+_S$, then  $u_i(x)\geq c(u) \mbox{d}(x)$, $i=1,\dots,m$ in $\Omega$ for some $c(u) >0$, and thus by \eqref{welldef} and $(g_1)$  the second derivative $\Phi_{uu}(u,\lambda)$ is well-defined on $\mathcal{W}\times  \mathcal{W}$ for every  $u \in \mathcal{W}^+_S$ .
\end{rem}
		 
\section{ On local existence of positive solutions for small $\lambda >0$}\label{sec:loc}

Introduce the \textit{nonlinear  Rayleigh quotient} (see \cite{ilSzulk})
\begin{equation}\label{lamb2}
\mathcal{R}(u):=\mathcal{R}(u,u)=\frac{\int |\nabla u|^2 dx-\int g_i(x,u)u_i dx}{\int |u|^{q} dx},~~u \in \mathcal{W}\setminus {0}.
\end{equation}
Note that for $u \in \mathcal{W}\setminus {0}$, $\lambda \in \mathbb{R}$, $\mathcal{R}(u)=\lambda$,  if and only if $\Phi'(u,\lambda)=0$, moreover  
\begin{equation}\label{NR}
	\mathcal{R}(u)=\lambda, ~\mathcal{R}'(u)>  0~~\Leftrightarrow \Phi'(u,\lambda)=0,~\Phi''(u,\lambda)> 0.
\end{equation}
Denote $S_1:=\{u \in \mathcal{W}: \|u\|_{\mathcal{W}}=1\}$.
\begin{prop}\label{Slambda}
There exists $\bar{\lambda}>0$ such that $\forall \lambda \in (0,\bar{\lambda})$,  	$\exists T(\lambda)>0$ such that $\mathcal{R}(T(\lambda)v)>\lambda$ and there holds: $ \forall v \in S_1$ there exists a unique $s_\lambda(v) \in (0, T(\lambda))$ such that   $\mathcal{R}(s_\lambda(v)v)=\lambda$ and $\mathcal{R}'(s_\lambda(v)v) >0$. Moreover, $\mathcal{R}(sv)<\lambda$ for $s \in (0,s_\lambda(v))$, $\mathcal{R}'(sv) >0$, $\forall s \in (0,T(\lambda))$.
\end{prop}
\begin{proof}
Compute
	\begin{equation*}\label{tmaxCC}
	 \mathcal{R}'(sv)=\frac{(2-q)s^{1-q}\int|\nabla {v}|^2 dx-\int\frac{\partial}{\partial s}( s^{-q}g_i(x,s{v})sv_i) dx}{\int |{v}|^{q} dx},~~s>0,~~v \in S_1.
	\end{equation*}
	Clearly,  $(g_1)$ and Sobolev's inequalities  yield that there exist constants $C_1,C_2>0$ which do not depend on $s>0$, $v \in S_1$ such that 
	$$
	\int\frac{\partial}{\partial s}( s^{-q}g_i(x,sv)sv_i) dx< C_1s^{\gamma_1-q-1}+C_2s^{\gamma_2-q-1}, ~~s>0, ~v \in S_1.
	$$
	Hence 
		\begin{align*}
		(2-q)s^{1-q}\int&|\nabla {v}|^2 dx-\int\frac{\partial}{\partial s}(s^{1-q}g_i(x,s{v})v_i) dx>r_1(t):=\\
		 &(2-q)s^{1-q}-C_1s^{\gamma_1-q-1}-C_2s^{\gamma_2-q-1},~~s>0, ~v \in S_1.
	\end{align*}
This implies that  there exists $s_0>0$ which does not depend on $v \in S_1$ such that  for any $s \in (0,s_0)$, 
$\mathcal{R}'(sv)>0$, $v \in S_1$. Similarly, 
\begin{equation*}
	\mathcal{R}(sv) > r_0(s):=C_0s^{2-q}-C_1's^{\gamma_1-q}-C_2's^{\gamma_2-q},~~\forall s>0, ~\forall v \in S_1,
\end{equation*}
for some $C_0,C_1',C_2'\in (0,\infty)$ which do not depend on $s>0$, $v \in S_1$.
Note that there is $s_1>0$ such that
$r_0'(s)>0$ for every $s \in (0,s_1)$. Let $\bar{s}=\min\{s_0,s_1\}$, and  introduce $\bar{\lambda}=r_0(\bar{s})$. 
Then by the above, $\mathcal{R}(sv) > r_0(s)$ and $\mathcal{R}'(sv) > 0$ for every $s \in (0,\bar{s})$, $\forall v \in S_1$.

Let $\lambda \in (0,\bar{\lambda})$. Since $r_0(s)$ is a monotone increasing function for $s \in (0, \bar{s})$ and $r_0(0)=0<\lambda<\bar{\lambda}=r_0(\bar{s})$, there exists a unique $T(\lambda) \in (0, \bar{s})$ such that $r_0(T(\lambda))=\lambda$. Note that for $v \in S_1$, $\mathcal{R}(0\cdot v)=0<\lambda=r_0(T(\lambda))<\mathcal{R}(T(\lambda)v)$ since $\mathcal{R}(sv) > r_0(s)$, $\forall s \in (0,\bar{s})$. This due to the monotonicity of $\mathcal{R}(sv)$ for $s \in (0, \bar{s})$, $\forall v \in S_1$ 
 implies  the existence of the unique $s_\lambda(v) \in (0, T(\lambda))$ such that $\mathcal{R}(s_\lambda(v)v)=\lambda$, $\mathcal{R}'(s_\lambda(v)v) >0$. Since $\mathcal{R}'(sv) >0$ for any $s \in (0,s_\lambda(v)]$, we have $\mathcal{R}(sv)<\lambda$ for $s \in (0,s_\lambda(v))$.
\end{proof}
 
\begin{lem}\label{cor444}
For any  $\lambda\in (0,\bar{\lambda})$, \eqref{p} admits  an asymptotically stable weak positive  solution $u_\lambda$. Moreover, $u_\lambda \in  \mathcal{W}^+_S$ and $\Phi(u_\lambda,\lambda)< 0$.
\end{lem}
\begin{proof}
Let $\lambda \in (0,\bar{\lambda})$. Take $T(\lambda)$ as in  Proposition \ref{Slambda}. 
Consider the following    Nehari manifold minimization problem
\begin{align}\label{MinConv1}
	&\hat{\Phi}_\lambda:=\min\{\Phi(u,\lambda):~~ u \in \mathcal{N}^{+,b}_\lambda\}.
	\end{align}
Here
\begin{align*}
	&\mathcal{N}^{+,b}_\lambda:=\{u \in \mathcal{N}_\lambda^+:~~\|u\|_{1,2}<T(\lambda)\}.
\end{align*}
Note that by \eqref{NR},
$$
\mathcal{N}^{+,b}_\lambda=\{u \in\mathcal{W}\setminus 0:~\mathcal{R}(u)= \lambda, \mathcal{R}'(u)> 0,~~\|u\|_{1,2}<T(\lambda)\}.
$$
  Proposition \ref{prop1} implies that $0>\hat{\Phi}_\lambda>-\infty$ for $\lambda \in (0,\bar{\lambda})$.

Let $(u_k) \subset \mathcal{N}^{+,b}_\lambda$ be a minimizing sequence of \eqref{MinConv1}, i.e., $\Phi(u_k,\lambda) \to \hat{\Phi}_\lambda$ as $k\to +\infty$. By Proposition \ref{prop1},  the minimizing sequence $(u_k)$ is bounded in $\mathcal{W}$ and thus, by the Banach–Alaoglu theorem it   contains a subsequence (again denoted by $(u_k)$) which weakly in $\mathcal{W}$ and strongly in $(L^p)^m$, $1\leq p<2^*$ converges to some limit point $u_\lambda$. The weak lower-semicontinuity of the norm of $\mathcal{W}$ and  Sobolev's embedding theorem imply that $\Phi(u,\lambda)$, $\mathcal{R}(u)$, $\mathcal{R}'(u)$ are weakly lower-semicontinuous on $\mathcal{W}$.
Hence, $\Phi(u_\lambda, \lambda)\leq \liminf_{k\to +\infty} \Phi(u_k, \lambda)=\hat{\Phi}_\lambda<0$, and therefore, $u_\lambda \neq 0$. 

Furthermore, $\sigma_\lambda:=\|u_\lambda\|_{1,2}\leq \liminf_{k\to +\infty}\|u_k\|_{1,2}\leq T(\lambda)$. Denote $v_\lambda:=u_\lambda/\|u_\lambda\|_{1,2}$. By Proposition \ref{Slambda}, there exists $s_\lambda(v_\lambda) \in (0, T(\lambda))$ such that $\mathcal{R}(s_\lambda(v_\lambda)v_\lambda)=\lambda$ and $\mathcal{R}'(s_\lambda(v_\lambda)v_\lambda) >0 $. Since $\mathcal{R}'(sv_\lambda)>0$,  $\forall s \in (0,T(\lambda))$ and  $\mathcal{R}(\sigma_\lambda v_\lambda)\equiv \mathcal{R}(u_\lambda)\leq \lambda$, we have  $\sigma_\lambda\leq s_\lambda(v_\lambda)$. Hence, 
	$$
	\Phi(s_\lambda(v_\lambda)v_\lambda,\lambda)\leq \Phi(\sigma_\lambda v_\lambda,\lambda)\leq \liminf_{k\to +\infty}\Phi(u_k, \lambda)=\hat{\Phi}_\lambda. 
	$$
	In view of that $s_\lambda(v_\lambda)v_\lambda \in \mathcal{N}^{+,b}_\lambda$, this implies that  $s_\lambda(v_\lambda)v_\lambda$ is a minimizer of \eqref{MinConv1}. Hence, $\Phi(u_\lambda, \lambda)\equiv \Phi(\sigma_\lambda v_\lambda,\lambda)=\hat{\Phi}_\lambda$, consequently $u_k \to u_\lambda$ strongly in $\mathcal{W}$, and thus, $u_\lambda$ is a minimizer of \eqref{MinConv1}. Moreover, since $u_\lambda=s_\lambda(v_\lambda)v_\lambda$ and $s_\lambda(v_\lambda)< T(\lambda)$, we get that $\|u_\lambda\|_{1,2}<T(\lambda)$. In addition,  Proposition \ref{Slambda} yields that $\mathcal{R}'(u_\lambda)> 0$. This implies that $u_\lambda$ weakly satisfies  \eqref{p}.  Since $\Phi(|u|, \lambda) =
\Phi(u, \lambda)$, $\Phi'(|u|, \lambda) =
\Phi'(u, \lambda)$, $\Phi''(|u|, \lambda) =
\Phi''(u, \lambda)$ for $u \in \mathcal{W}$, where $|u|:=(|u_1|, \ldots, |u_m|))$, $u \in \mathbb{R}^m$, we may
assume that $u_{\lambda,i} \geq 0$ in $\Omega$, $i = 1,\ldots, m$, and therefore (see Remark \ref{rem1}), $u_\lambda $ is a positive solution of \eqref{p} and  $u_\lambda \in  \mathcal{W}^+_S$.  Since $u_\lambda$ is a local minimizer of \eqref{MinConv1} and $\Phi''(u_\lambda,\lambda) \equiv\Phi_{uu}(u_\lambda,\lambda)(u_\lambda,u_\lambda)>0$, we have  $\Phi_{uu}(u_\lambda, \lambda)(\phi,\phi)> 0$,~$\forall \phi \in \mathcal{W}$, and thus,    $u_\lambda$ is an asymptotically stable  solution of \eqref{p}. 
\end{proof}

\section{Proof
 of Theorem  \ref{thmM}} \label{sec: profthm1}

\begin{lem}\label{lem10} 
Suppose  $(g_1)$, $(g_4)$, $1<q<2$. Then 
$\lambda^{*}_S<+\infty$.
\end{lem}
\begin{proof} Let $\Omega^+$ be the subdomain from hypothesis $(g_4)$. 
Without loss of generality, we may assume that $\partial \Omega^+$ is a $C^1$-manifold. 
Denote	by $\phi_1(\Omega^+)$ the eigenfunction of $(-\Delta)$ in $\W_2(\Omega^+)$ corresponding to $\lambda_1(\Omega^+)$. It is well-known that $\phi_1(\Omega^+) \in C^1(\overline{\Omega^+})$ and $\phi_1(\Omega^+)> 0$ in $\Omega^+$. Hence, $\phi_1(\Omega^+) \in \Sigma(u)$,  $\forall u \in \mathcal{W}^+_S$. Moreover  
\begin{align*}
 \int (  \nabla u_i, \nabla \phi_1(\Omega^+) )\leq \lambda_1(\Omega^+) \int u_i	\phi_1(\Omega^+),~~~~\forall u\in \mathcal{W}^+_S,~~i=1,\ldots,m.
\end{align*}
This by \eqref{MainB},  implies 
\begin{align}\label{LargL}
	\lambda^{*}_S\leq \sup_{u\in \mathcal{W}^+_S}\mathcal{R}(u, \bar{\phi}_1)&\leq \sup_{u\in \mathcal{W}^+_S}\frac{\sum_{i=1}^m\int (\lambda_1(\Omega^+) u_i	-g_i(x,u))\phi_1(\Omega^+)}
	{\sum_{i=1}^m\int_{\Omega^+} u_i^{q-1}\phi_1(\Omega^+) }.
\end{align}
 Here $\bar{\phi}_1:=(\phi_1(\Omega^+),\ldots,\phi_1(\Omega^+))$. 
Consider  
$$
h(x, u):=\frac{\sum_{i=1}^m( \lambda_1(\Omega^+)u_i -g_i(x,u))}{\sum_{i=1}^mu_i^{q-1}}, ~~u \in (\mathbb{R}^+)^m, ~~x \in \Omega^+.
$$
By $(g_4)$, $\sup_{x\in \Omega^+}h(x, u) \to -\infty$~as~$|u| \to +\infty$, 
and by $(g_1)$, 
  $\sup_{x\in \Omega^+}h(x, u) \downarrow  0$  as $|u| \to 0$. Hence, $\Lambda:=\sup_{x\in \Omega^+}\max_{u\in (\mathbb{R}^+)^m}h(x,u)>0$. This by  \eqref{LargL} yields  $\lambda^{*}_S<\Lambda<+\infty$.

\end{proof}

From Lemma \ref{cor444} it follows that \eqref{p} has a weak  solution $u_\lambda \in \mathcal{W}^+_S$ for any  $\lambda\in (0,\bar{\lambda})$.  This yields that
$\inf_{v\in \Sigma(u)}\mathcal{R}(u_\lambda, v)=\lambda$, $\forall  \lambda \in (0, \bar{\lambda})$,  and thus,  $\lambda^{*}_S>\lambda>0$. Hence, we have proved that  $0<\lambda^{*}_S<+\infty$.

Proofs of assertions $(1^o)$ and $(2^o)$ of Theorem \ref{thmM} follow from  $(1^o)$, $(2^o)$ of Lemma \ref{lemM}, respectively.

By $(2^o)$ of Lemma \ref{lemM}  the principal eigenvalue $\delta(u_{\lambda^{*}_S})$ of $\Phi_{uu}(u_{\lambda^{*}_S}, {\lambda^{*}_S})$ is non-negative.
Suppose, contrary to our claim, that $\delta(u_{\lambda^{*}_S})>0$. Then  $\Phi_{uu}(u_{\lambda^{*}_S}, {\lambda^{*}_S}):\mathcal{W} \to  \mathcal{W}'$ is nonsingular linear operator, and  so it is $\Phi_{uu}(u_{\lambda^{*}_S}, {\lambda^{*}_S}):X \to  \mathcal{W}'$, where 
\begin{align*}
	X=\{u \in (C^{1}(\overline{\Omega}))^m:~\|u\|_X=\|u\|_{(C^{1}(\overline{\Omega}))^m}<\infty\}. 
\end{align*}
 From \eqref{welldef} and $(g_1)$ it follows that there exist neighbourhoods $U$ of $u_{\lambda^{*}_S}$ in $X$ and $V$ of $\lambda^{*}_S$ in $\mathbb{R}$ such that  $\Phi_{uu}(u,\lambda) \in C(V\times U; \mathcal{L}(X,\mathcal{W}'))$, where $\mathcal{L}(X,\mathcal{W}'))$ denotes the Banach space of bounded linear operators
from $X$ into  $\mathcal{W}'$. Hence, by Implicit Functional Theorem (see, e.g, \cite{kielh}) there is a neighbourhood $V_1\times U_1 \subset V\times U$  of $(\lambda^{*}_S, u_{\lambda^{*}_S})$ in $\mathbb{R}\times X$ and a mapping $V_1 \ni \lambda \mapsto u_\lambda \in U_1$ 
such that $u_\lambda|_{\lambda=\lambda^{*}_S}=u_{\lambda^{*}_S}$ and $\Phi_{u}(u_{\lambda}, \lambda)=0$, $\forall \lambda \in V_1$. 
Furthermore, the map $u_{(\cdot)}: V_1 \to X$ is continuous. Since $\delta(u_{\lambda^{*}_S})>0$, this implies that there is a neighbourhood $V_2 \subset V_1$ of $\lambda^{*}_S$ such that $\delta(u_\lambda)>0$ for every $\lambda \in V_2$. However, this contradicts assertion $(2^o)$ of Theorem \ref{thmM}, and thus, we get that $\delta(u_{\lambda^{*}_S})=0$.  This conclude the proof of Theorem \ref{thmM}.

\section{Appendix: Ekeland's principal for  smooth functionals}\label{sec:appendix}
Let $W$ be a Banach space and  $V\subset W$ be an open  domain. Denote  $B_r:=\{\phi \in W: ~~\|\phi\|_{W}\leq r\}$, $r>0$.

Assume that $F: V \to \mathbb{R}$, $F \in C^2(V)$. 
Consider
\begin{equation}\label{ApMin}
	\hat{F}=\inf_{v \in V}F(v).
\end{equation}

\begin{thm}\label{thm:Ek}
Assume that $|	\hat{F}|<+\infty$. Suppose that there exist $\tau_0, a_0, C_0\in (0,+\infty)$, and a minimizing sequence $(v_k) \subset V$ of \eqref{ApMin} such that  
\begin{align}
	&\|F_{vv}(v_k+\tau \phi)\|_{(W\times W)'} < 	\frac{C_0}{(1-|\tau| a_0)^2}<+\infty,\label{DDRstG}\\
&v_k+\tau \phi \in V,~~ \forall \tau \in (-\tau_0,\tau_0), ~~\forall \phi \in B_1,~~\forall k=1,\ldots.\label{DDRstG2}
\end{align}
 Then 
$$
\|F_{v}(v_k)\|_{W'}	:=\sup_{\xi \in W\setminus 0}\frac{|F_{v}(v_k)(\xi)|}{\|\xi\|_W} \to 0~~~\mbox{as}~~~ k \to +\infty.
$$
\end{thm}
\begin{proof} 
Suppose, contrary to our claim, that there exists $\alpha>0$ such that
$$
\|F_{v}(v_k)\|_{W'}>\alpha,~~~\forall k=1,\ldots. 
$$
Then for every $k=1,\ldots$, there exists $\phi_k \in V$, $\|\phi_k\|_{W}=1$ such that
$|F_{v}(v_k)(\phi_k)|>\alpha$. 
By the Taylor expansion  
$$
F(v_k +\tau \phi_k)=F(v_k)+\tau F_{v}(v_k)(\phi_k)+\frac{\tau^2}{2}F_{vv}(v_k+\theta_k \tau \phi_k)(\phi_k,\phi_k),
$$
for sufficiently small $|\tau|$, and some $\theta_k \in (0,1)$, $k=1,\ldots$. Suppose, for definiteness, that  $F_{v}(v_k)(\phi_k)>\alpha$. Then for $\tau\in (-\tau_0, 0)$, by \eqref{DDRstG}
\begin{align*}
	F(v_k +\tau \phi_k)\leq F(v_k)+\tau \alpha +\frac{\tau^2}{2}\frac{C_0}{(1+\tau a_0)^2}, ~~~k=1,\ldots. 
\end{align*}
It is easily seen that there exists $\tau_{1}\in (0,\tau_0)$  such that  
$$
\kappa(\tau):=\tau \left( \alpha +\frac{\tau}{2}\frac{C_0}{(1+\tau a_0)^2}\right) <0,~~~~\forall \tau \in (-\tau_{1}, 0).
$$
Since $(v_k)$ is a minimizing sequence, for any $\varepsilon>0$ there exists $k(\varepsilon)$ such that 
$$
F(v_k)<\hat{F}+\varepsilon,~~\forall k>k(\varepsilon).
$$
Take $\tau \in (-\tau_1, 0)$ and  $\varepsilon_0=-\kappa(\tau)/2$. Then by above, $\forall k>k(\varepsilon_0)$   we get
$$
F(v_k +\tau \phi_k)<\hat{F}+\varepsilon_0+\kappa(\tau) = \hat{F}+\kappa(\tau)/2<\hat{F},
$$
and thus, by \eqref{DDRstG2}, a contradiction.
\end{proof}

\end{document}